\documentclass[journal]{IEEEtran}

\usepackage{graphicx}
\usepackage{graphics} 
\usepackage{amsmath} 
\usepackage{amssymb}  
\usepackage{amsfonts}
\usepackage[ruled,vlined]{algorithm2e}
\usepackage{pseudocode}
\usepackage{tabu}
\usepackage{color}
\usepackage{cite}

\usepackage{enumitem}
\usepackage{epsfig}
\usepackage{mathrsfs}
\usepackage{array}
\usepackage{amsfonts,booktabs}
\usepackage{lipsum,amsmath,multicol}
\usepackage{graphicx}
\usepackage{subfigure}
\usepackage{times}
\usepackage{pdfsync}
\usepackage{pgfplots}
\usepackage{pgfplotstable}
\usepackage[subnum]{cases}
\usepackage{filecontents}
\usepackage{multirow}
\usepackage{float}
\usepackage{setspace}
\usepackage{mathtools}

\newtheorem{theorem}{\bf Theorem}

\newtheorem{remark}{\bf Remark}
\newtheorem{definition}{\bf Definition}
\newtheorem{claim}{\bf Claim}

\newtheorem{lemma}{\bf Lemma}

\def\QED{~\rule[-1pt]{5pt}{5pt}\par\medskip}
\newenvironment{proof}{{\bf Proof: \ }}{ \hfill \QED}

\DeclareMathOperator*{\argmin}{arg\,min}

\newcommand*{\LONGVERSION}{}
\newcommand{\upd}{\color{black}}

\newcommand{\beq}{\begin{equation}}
\newcommand{\eeq}{\end{equation}}
\newcommand{\beqa}{\begin{eqnarray}}
\newcommand{\eeqa}{\end{eqnarray}}

\newcommand{\paren}[1]{\left(#1\right)}

\newcommand{\abs}[1]{\left|#1\right|} 
\newcommand{\I}[1]{\ensuremath{\mathsf{1}{\left\{#1\right\}}}} 
\newcommand{\PRP}[1]{\ensuremath{\mathsf{Pr}\left(#1\right)}} 
\newcommand{\ES}[1]{\ensuremath{\mathbb{E}\left[#1 \right]}} 

\newcommand{\logp}[1]{\ensuremath{\log\paren{#1}}}

\begin{document}
%
\title{Linearly Solvable \\ Mean-Field Traffic  Routing Games$^*$}
%

%
%

\author{
 \and Takashi Tanaka$^{1}$ \hspace{3ex} \and Ehsan Nekouei$^{2}$  \hspace{3ex} \and Ali Reza Pedram$^{3}$ \hspace{3ex} \and 
Karl Henrik Johansson$^{4}$ 
\thanks{
$^{*}$ A preliminary version of this work has been presented at \cite{tanaka2018linearly}. 
\ifdefined\LONGVERSION
\else
In the interest of page limitations, proofs of technical results in this paper are partly deferred to  \cite{arXivversion}.
\fi
}
\thanks{
$^{1, 3}$University of Texas at Austin, TX, USA.
        {\tt\small \{ttanaka, apedram\}@utexas.edu}. 
$^{2}$City University of Hong Kong, Kowlong Tong, Hong Kong.
        {\tt\small ekouei@cityu.edu.hk.}
$^{4}$KTH Royal Institute of Technology, Stockholm, Sweden.
        {\tt\small \{kallej\}@kth.se}. 
}
}

\maketitle

\begin{abstract}
We consider a dynamic traffic routing game over an urban road network involving a large number of drivers in which each driver selecting a particular route is subject to a penalty that is affine in the logarithm of the number of drivers selecting the same route. 
We show that the mean-field approximation of such a game leads to the so-called linearly solvable Markov decision process, implying that its mean-field equilibrium (MFE) can be found simply by solving a finite-dimensional linear system backward in time.
Based on this backward-only characterization, it is further shown that the obtained MFE has the notable property of strong time-consistency. A connection between the obtained MFE and a particular class of fictitious play is also discussed.
\end{abstract}
%
%
\section{Introduction}

The mean-field game (MFG) theory, introduced by the authors of   \cite{caines2013mean} and \cite{lasry2007mean} almost concurrently, provides a powerful framework to study stochastic dynamic games where 
(i) the number of players involved in the game is large, (ii) each individual player's impact on the network is infinitesimal, and (iii) players' identities are indistinguishable. 
The central idea of the MFG theory is to approximate, in an appropriate sense, the original large-population game problem by a single-player optimal control problem, in which individual player's best response to the mean field (average behavior of the population) is analyzed.
Typically, the solution to the latter problem is characterized by a pair of  backward Hamilton-Jacobi-Bellman (HJB) and forward Fokker-Planck-Kolmogorov (FPK) equations; the HJB equation guarantees player-by-player optimality, while the FPK equation guarantees time consistency of the solution.
The coupled HJB-FPK systems, as well as alternative mathematical characterizations (e.g., McKean-Vlasov systems), have been studied extensively \cite{lasry2007mean,achdou2010mean,carmona2013probabilistic}.

There has been a recent growth in the literature on MFGs and its applications.
MFGs under Linear Quadratic (LQ) \cite{Huang2010,HLW2016,MT2017} and more general settings \cite{Huang2012, CLM2015} are both extensively explored.
MFGs with a major agent and a large number of minor agents are studied \cite{Huang2012} and applied to design decentralized security defense decisions in a mobile ad hoc network \cite{WYTH2014}. MFGs with multiple classes of players are investigated in \cite{TH2011}. The authors of \cite{BTB2016} studied the existence of robust (minimax) equilibrium in a class of stochastic dynamic games.
In \cite{ZTB2011}, the authors analyzed the equilibrium of a hybrid stochastic game in which the dynamics of agents are affected by continuous disturbance as well as random switching signals.
Risk-sensitive MFGs were considered in  \cite{tembine2014risk}.
While continuous-time continuous-state models are commonly used in the references above,  
\cite{jovanovic1988anonymous, weintraub2006oblivious,gomes2010discrete,salhab2018mean,saldi2018markov} have considered the MFG in discrete-time and/or discrete-state regime.
The issues of time inconsistency in MFG and mean-field type optimal control problems are discussed in \cite{bensoussan2013linear,djehiche2016characterization,cisse2014cooperative}.

While substantial progress has been made on the MFG literature in recent years, there has been a long history of mean-field-like approaches to large-population games in the transportation research literature \cite{sheffi1984urban}. 
A well-known consequence of a mean-field-like analysis of the traffic user equilibrium is the Wardrop's first principle  \cite{wardrop1952some,correa2011wardrop}, which provides the following characterization of the traffic condition at an equilibrium: \emph{journey times on all the routes actually
used are equal, and less than those which
would be experienced by a single vehicle on
any unused route}. This result, as well as a generalized concept known as \emph{stochastic user equilibrium} (SUE) \cite{daganzo1977stochastic}, has played a major role in the transportation research, including the convergence analysis of users' day-to-day routing policy adjustment process 
\cite{fisk1980some,dial1971probabilistic,powell1982convergence,sheffi1982algorithm,liu2009method}.
However, currently only a limited number of results are available connecting the transportation research and recent progress in the MFG theory. The work \cite{salhab2018mean} considers discrete-time discrete-state mean-field route choice games.
In \cite{CLM2015}, the authors modeled the interaction between drivers on a straight road as a non-cooperative game and characterized its MFE.
In \cite{BZP2017}, the authors considered a continuous-time Markov chain to model the aggregated behavior of drivers on a traffic network. {\upd A Markovian framework for traffic assignment problems is introduced in \cite{baillon2008markovian}, which is similar to the problem formulation adopted in this paper. A connection between large-population Markov Decision Processes (MDPs) and MFGs has been discussed in a recent work \cite{yu2019primal}.} MFG has been applied to pedestrian crowd dynamics modeling in \cite{lachapelle2011mean,dogbe2010modeling}.

In this paper, we apply the MFG theory to study the strategic behavior of infinitesimal drivers traveling over an urban traffic network. Specifically, we consider a discrete-time dynamic stochastic game wherein, at each intersection, each driver randomly selects one of the outgoing links as her next destination according to a randomized policy. We assume that individual drivers' dynamics are decoupled from each other, while their cost functions are coupled.
In particular, we assume that the cost function for each driver is congestion-dependent, and is affine in the logarithm of the number of drivers taking the same route. 
We regard the congestion-dependent term in the cost function as an incentive mechanism (toll charge) imposed by the Traffic System Operator (TSO). 
Although the assumed structure of cost functionals is restrictive, the purpose of this paper is to show that the considered class of MFGs exhibits a \emph{linearly solvable} nature, and requires somewhat different treatments from the standard MFG formalism.
We emphasize that the computational advantages that follow from this special property are notable both from the existing MFG and the transportation research perspectives. 
Contributions of this paper are summarized as follows:
\begin{enumerate}[leftmargin=*]
\item Linear solvability: We prove that the MFE of the game described above is given by the solution to a linearly solvable MDP \cite{todorov2007linearly}, meaning that it can be computed by performing a sequence of matrix multiplications backward in time \emph{only once}, without any need of forward-in-time computations. 
This offers a tremendous computational advantage over the conventional characterization of the MFE where there is a need to solve a forward-backward HJB-FPK system, which is often a non-trivial task \cite{achdou2010mean}.
\item Strong time-consistency: Due to the backward-only characterization, the MFE in our setting is shown to be \emph{strongly time-consistent}  \cite{basar1999dynamic}, a stronger property than what follows from the standard forward-backward characterization of MFEs.
\item MFE and fictitious play: With an aid of numerical simulation, we show that the derived MFE can be interpreted as a limit point of the belief path of the \emph{fictitious play} process \cite{monderer1996fictitious} in a scenario where the traffic routing game is repeated.
\end{enumerate}
The rest of the paper is organized as follows: The traffic routing game is set up in Section~\ref{secprob} and its mean field approximation is discussed in Section~\ref{secmfapprox}. The linearly solvable MDPs are reviewed in Section~\ref{secoptcontrol}, which is used to derive the MFE of the traffic routing game in Section~\ref{secmfe}. 
Time consistency of the derived MFE is studied in Section~\ref{sectc}. A connection between MFE and fictitious play is investigated in Section~\ref{secfictitious}. Numerical studies are summarized in Section~\ref{secsimulation} before we conclude in Section~\ref{secconclude}.

\section{Problem Formulation}
\label{secprob}

The traffic game studied in this paper is formulated as an $N$-player, $T$-stage dynamic game. 
Denote by $\mathcal{N}=\{1, 2, \cdots , N\}$ the set of players (drivers) and by $\mathcal{T}=\{0, 1, \cdots , T-1\}$ the set of time steps at which players make decisions.

\subsection{Traffic graph}
The \emph{traffic graph} is a directed graph $\mathcal{G}=(\mathcal{V},\mathcal{E})$, where $\mathcal{V}=\{1, 2, ... , V\}$ is the set of nodes (intersections) and $\mathcal{E}=\{1, 2, ... , E \}$ is the set of directed edges (links). For each $i\in \mathcal{V}$, denote by $\mathcal{V}(i) \subseteq \mathcal{V}$ the set of intersections to which there is a directed link from the intersection $i$.   
At any given time step $t\in \mathcal{T}$, each player is located at an intersection.
The node  at which the $n$-th player is located at time step $t$ is denoted by $i_{n,t}\in \mathcal{V}$.  
At every time step, player $n$ at location $i_{n,t}$ selects her next destination $j_{n,t}\in \mathcal{V}(i_{n,t})$.
By selecting $j_{n,t}$ at time $t$, the player $n$ moves to the node $j_{n,t}$ at time $t+1$ deterministically (i.e., $i_{n,t+1}=j_{n,t}$).

\subsection{Routing policy}

At every time step $t$, each player selects her next destination according to a randomized routing policy.
Let $\Delta^J$ be the $J$-dimensional probability simplex, and 
$Q_{n,t}^i=\{Q_{n,t}^{ij}\}_{j\in \mathcal{V}(i)}\in \Delta^{|\mathcal{V}(i)|-1}$ be the probability distribution according to which player $n$ at intersection $i$ selects the next destination $j \in \mathcal{V}(i)$.
We consider the collection $Q_{n,t}=\{Q_{n,t}^i\}_{i\in \mathcal{V}}$ of such probability distributions as the \emph{policy} of player $n$ at time $t$.
For each $n\in \mathcal{N}$ and $t\in \mathcal{T}$, notice that  $Q_{n,t}\in \mathcal{Q}$, where
\[
\mathcal{Q}=\left\{\{Q^i\}_{i\in\mathcal{V}}: Q^i \in \Delta^{|\mathcal{V}(i)|-1} \;\; \forall i\in\mathcal{V} \right\}
\]
is the space of admissible policies. 
Suppose that the initial locations of players $\{i_{n,0}\}_{n\in\mathcal{N}}$ are independent and identically distributed random variables with
 $P_{n,0}=P_{0}\in\Delta^{|\mathcal{V}|-1}$.
Note that if the policy $\{Q_{n,t}\}_{t\in\mathcal{T}}$ of player $n$ is fixed, then the probability distribution $P_{n,t}=\{P_{n,t}^i\}_{i\in \mathcal{V}}$ of her location at time $t$ is  computed recursively by
\begin{equation}
\label{eqdyn}
P_{n,t+1}^j=\sum_i P_{n,t}^i Q_{n,t}^{ij} \;\; \forall t\in\mathcal{T}, j \in\mathcal{V}.
\end{equation}
If $(i_{n,t}, j_{n,t})$ is the location-action pair of player $n$ at time $t$, it has the joint distribution $P_{n,t}^i Q_{n,t}^{ij}$.
We assume that location-action pairs $(i_{n,t}, j_{n,t})$ and $(i_{m,t}, j_{m,t})$ for two different players $m\neq n$ are drawn independently under individual policies $\{Q_{n,t}\}_{t\in\mathcal{T}}$ and $\{Q_{m,t}\}_{t\in\mathcal{T}}$.
With a slight abuse of notation, we sometimes write $Q_n:=\{Q_{n,t}\}_{t\in\mathcal{T}}$ for simplicity.

\subsection{Cost functional}
We assume that, at each time step, the cost functional for each player has two components as specified below:
\subsubsection{Travel cost}
For each $i\in\mathcal{V}$, $j\in\mathcal{V}(i)$ and $t\in\mathcal{T}$, let $C_t^{ij}$ be a given constant representing the cost (e.g., fuel cost) for every player selecting $j$ at location $i$ at time $t$. 
\subsubsection{Tax cost}
We assume that players are also subject to individual and time-varying tax penalties calculated by the TSO. 
{\upd The tax charged to player $n$ at time step $t$ depends not only on her own location-action pair at $t$, but also on the behavior of the entire population at that time step.}
 Specifically, we consider the log-population tax mechanism, where the tax charged to player $n$ taking action $j$ at location $i$ at time $t$ is
\begin{equation}
\label{eqtax}
\pi_{N,t,n}^{ij}=  \alpha  \left( \log \frac{K_{N,t}^{ij}}{K_{N,t}^i }-\log R_t^{ij} \right).
\end{equation}
Here, $\alpha>0$ is a fixed constant characterizing the ``aggressiveness'' of the tax mechanism. 
In \eqref{eqtax}, $K_{N,t}^i$ is the number of players  (including player $n$) who are located at the intersection $i$ at time $t$. Likewise, $K_{N,t}^{ij}$ is the number of players (including player $n$) who takes the action $j$ at the intersection $i$ at time $t$.
The parameters $R_t^{ij}> 0$ are fixed constants satisfying
$\sum_j R_t^{ij}=1$ for all $i$. We interpret $R_t^{ij}$ as the ``reference'' routing policy specified by the TSO in advance. Notice that \eqref{eqtax} indicates that agent $n$ receives a positive reward by taking action $j$ at location $i$ at time $t$ if $K_{N,t}^{ij}/K_{N,t}^{i} < R_t^{ij}$ {\upd (i.e., the realization of the traffic flow is below the designated congestion level)}, while she is penalized by doing so if $K_{N,t}^{ij}/K_{N,t}^{i} > R_t^{ij}$. Since $K_{N,t}^{i}$ and $K_{N,t}^{ij}$ are random variables, $\pi_{N,t,n}^{ij}$ is also a random variable.
We assume that the TSO is able to observe $K_{N,t}^{i}$ and $K_{N,t}^{ij}$ at every time step so that $\pi_{N,t,n}^{ij}$ is computable.\footnote{Whenever $\pi_{N,t,n}^{ij}$ is computed, we have both $K_{N,t}^{ij}\geq 1$ and $K_{N,t}^{i}\geq 1$ since at least player $n$ herself is counted. Hence \eqref{eqtax} is well-defined.}
In what follows, we assume that each player is risk neutral.
{\upd That is, each player is interested in choosing a policy that minimizes the expected sum of travel and tax costs incurred over the planning horizon $\mathcal{T}$.
For player $n$ whose location-action pair at time step $t$ is $(i,j)$,  the expected tax cost incurred at that time step can be expressed as 
\begin{equation}
\label{eqexp_pi_general}
\Pi_{N,n,t}^{ij}\triangleq \mathbb{E}\left[\pi_{N,n,t}^{ij} \mid  i_{n,t}=i, j_{n,t}=j \right].
\end{equation}
\ifdefined\LONGVERSION
As we detail in equation \eqref{eqtijgeneral} in Appendix~\ref{appbinomial}, for each location-action pair $(i,j)$, $\Pi_{N,n,t}^{ij}$ can be expressed in terms of $Q_{-n}\triangleq\{Q_m\}_{m\neq n}$.
\else
As we detail in \cite[Appendix A, equation (22)]{arXivversion}, for each location-action pair $(i,j)$, $\Pi_{N,n,t}^{ij}$ can be expressed in terms of $Q_{-n}\triangleq\{Q_m\}_{m\neq n}$.
\fi
The fact that $\Pi_{N,n,t}^{ij}$ does not depend on player $n$'s own policy will be used to analyze the optimal control problem \eqref{eqprob_n} below.\footnote{Although the value of $\Pi_{N,n,t}^{ij}$ for each $(i,j)$ cannot be altered by player $n$'s policy, she can minimize the total cost by an appropriate route choice (e.g., by avoiding links with high toll fees).}}


\subsection{Traffic routing game}
\label{sectrafficgame}
Overall, the cost functional to be minimized by the $n$-th player in the considered  game is given by
\begin{equation}
J\left(Q_n, Q_{-n}\right)=\sum_{t=0}^{T-1} \sum_{i,j} P_{n,t}^i Q_{n,t}^{ij}\left(C_t^{ij}+\Pi_{N,n,t}^{ij}\right). \label{eqobjn}
\end{equation}
Notice that this quantity depends not only on the $n$-th player's own policy $Q_n$ but also on the other players' policies $Q_{-n}$ through the term $\Pi_{N,n,t}^{ij}$. 
Equation \eqref{eqobjn} defines an $N$-player dynamic game, which we call the \emph{traffic routing game} hereafter. 
We introduce the following equilibrium concepts.
\begin{definition}
The $N$-tuple of strategies $\{Q_{n}^*\}_{n\in\mathcal{N}}$ is said to be a \emph{Nash equilibrium} if the inequality $J\left(Q_n, Q_{-n}^*\right) \geq J\left(Q_n^*, Q_{-n}^*\right)$ holds for each $n\in \mathcal{N}$ and $Q_n$.
\end{definition}
\begin{definition}
The $N$-tuple of strategies $\{Q_{n}^*\}_{n\in\mathcal{N}}$ is said to be \emph{symmetric} if $Q_1^*=Q_2^*=\cdots=Q_N^*$.
\end{definition}
\begin{remark}
The $N$-player game described above is a \emph{symmetric game} in the sense of \cite{cheng2004notes}. Thus, \cite[Theorem~3]{cheng2004notes} is applicable to show that it has a symmetric Nash equilibrium.
\end{remark}
\begin{remark}
We assume that players are able to compute a Nash equilibrium strategy $\{Q_{n}^*\}_{n\in\mathcal{N}}$ prior to the execution of the game based on the public knowledge $\mathcal{G}, \alpha, \mathcal{N}, \mathcal{T}, R_t^{ij}, C_t^{ij}$ and $P_{0}$. Often the case, it is favorable that a Nash equilibrium is \emph{time-consistent} in that no player is given an incentive to deviate from the precomputed equilibrium routing policy after observing
real-time data (such as  $K_{N,t}^{i}$ and $K_{N,t}^{ij}$).
In Section~\ref{sectc}, we discuss a notable time consistency property of an equilibrium of the traffic routing game formulated above in the large-population limit $N\rightarrow \infty$.
\end{remark}

\section{Mean Field Approximation}
\label{secmfapprox}

In the remainder of this paper, we are concerned with the large-population limit $N\rightarrow \infty$ of the traffic routing game.
\begin{definition}
A set of strategies $\{Q_{n}^*\}_{n\in\mathcal{N}}$ is said to be an \emph{MFE} if the following conditions are satisfied.
\begin{itemize}
\item[(a)] It is symmetric, i.e., $Q_1^*=Q_2^*=\cdots=Q_N^*$.
\item[(b)] There exists a sequence $\epsilon_N$ satisfying $\epsilon_N \searrow 0$ as $N\rightarrow \infty$ such that for each $n \in \mathcal{N}=\{1, 2, ... , N\}$ and $Q_n$, the inequality $J(Q_n,Q_{-n}^*)+\epsilon_N \geq J(Q_n^*,Q_{-n}^*)$ holds.
\end{itemize}
\end{definition}
{\upd 
Now, we derive a condition that an MFE must satisfy by analyzing player $n$'s  best response when all other players adopt a homogeneous routing policy $Q^*=\{Q_t^*\}_{t\in\mathcal{T}}$. Since $Q^*$ is adopted by all players other than $n$,  the probability that a specific player $m(\neq n)$ is located at $i$ is given by $P_t^{i*}$, where $P^*=\{P_t^*\}_{t\in\mathcal{T}}$ is computed recursively by
 \[
P_{t+1}^{j*}=\sum_i P_t^{i*} Q_t^{ij*} \;\; \forall  j\in\mathcal{V}.
\]
Player $n$'s best response  is characterized by the solution to the following optimal control problem:
\begin{equation}
\label{eqprob_n}
\min_{\{Q_{t}\}_{t\in\mathcal{T}}} \sum_{t=0}^{T-1} \sum_{i,j} P_{t}^i Q_{t}^{ij}\left(C_t^{ij}+\Pi_{N,n,t}^{ij}\right).
\end{equation}
Here, we note that $\Pi_{N,n,t}^{ij}$ is fully determined by the homogeneous policy $Q^*$ adopted by all other players.
\ifdefined\LONGVERSION
(The detail is shown in equation \eqref{eqtjieq} in Appendix~\ref{appbinomial}.)
\else
(The detail is shown in \cite[Appendix A, equation (23)]{arXivversion}.)
\fi
In \eqref{eqprob_n}, we wrote $P_{t}$ and $Q_{t}$ in place of $P_{n,t}$ and $Q_{n,t}$ to simplify the notation.

To analyze player $n$'s best response when $N\rightarrow \infty$, we compute the quantity $\lim_{N\rightarrow \infty} \Pi_{N,n,t}^{ij}$ as follows:}
\begin{lemma}
\label{lemlimpi}
Let $\Pi_{N,n,t}^{ij}$ be defined by \eqref{eqexp_pi_general}. If $Q_{m,t}=Q_t^*$ for all $m\neq n$ and $P_t^{i*}Q_t^{ij*}>0$, then
\[
\lim_{N\rightarrow \infty} \Pi_{N,n,t}^{ij}=\alpha \log \frac{Q_t^{ij*}}{R_t^{ij}}.  
\]
\end{lemma}
\begin{proof}
\ifdefined\LONGVERSION
Appendix~\ref{app1}.
\else
\cite[Appendix B]{arXivversion}.
\fi
\end{proof}
Intuitively, Lemma~\ref{lemlimpi} shows that the optimal control problem \eqref{eqprob_n} when $N$ is large is ``close to'' the optimal control problem:
\begin{equation}
\min_{\{Q_{t}\}_{t\in\mathcal{T}}}  \sum_{t=0}^{T-1} \sum_{i,j} \!P_{t}^i Q_{t}^{ij}\!\left(\!C_t^{ij}\!+\!\alpha\log\frac{Q_t^{ij*}}{R_t^{ij}} \right). \label{eqprob_n_inf}
\end{equation}
In order for the policy $Q^*$ to constitute an MFE, the policy $Q^*$ itself needs to be the best response by player $n$. In particular, $Q^*$ must solve the optimal control problem \eqref{eqprob_n_inf}. That is, the following {\upd fixed point condition} must be satisfied:
\begin{equation}
Q^*\in \argmin\limits_{\{Q_{t}\}_{t\in\mathcal{T}}}  \sum_{t=0}^{T-1} \sum_{i,j} \!P_{t}^i Q_{t}^{ij}\!\left(\!C_t^{ij}\!+\!\alpha\log\frac{Q_t^{ij*}}{R_t^{ij}} \right). \label{eqconsistency}
\end{equation}
In the next two sections, we show that the condition \eqref{eqconsistency} is closely related to the class of optimal control problems known as \emph{linearly-solvable MDPs} \cite{todorov2007linearly,dvijotham2011unified}. Based on this observation, we show that an MFE can be computed efficiently.

\section{Linearly Solvable MDPs}
\label{secoptcontrol}
In this section, we review linearly-solvable MDPs  \cite{todorov2007linearly,dvijotham2011unified} and their solution algorithms. For each $t\in\mathcal{T}$, let $P_t$ be the probability distribution over $\mathcal{V}$ that evolves according to
\begin{equation}
\label{eqpq2}
P_{t+1}^j=\sum_i P_t^i Q_t^{ij} \;\; \forall  j\in\mathcal{V}
\end{equation}
with the initial state $P_0$.
We assume $C_t^{ij}$, $R_t^{ij}$ for each $t\in\mathcal{T}, i\in\mathcal{V}, j\in\mathcal{V}$ and $\alpha$ are given positive constants.
Consider the $T$-step optimal control problem:
\begin{equation}
\min_{\{Q_t\}_{t\in\mathcal{T}}}\sum_{t=0}^{T-1}\sum_{i,j}P_t^i Q_t^{ij}\left(C_t^{ij}+\alpha \log\frac{ Q_t^{ij}}{R_t^{ij}} \right).
\label{eqoptcontrol2}
\end{equation}
The logarithmic term in \eqref{eqoptcontrol2} can be written as the Kullback--Leibler (KL) divergence from the reference policy $R_t^{ij}$ to the selected policy $Q_t^{ij}$. For this reason \eqref{eqoptcontrol2} is also known as the \emph{KL control} problem \cite{theodorou2010generalized}. Notice the similarity and difference between the optimal control problems \eqref{eqprob_n_inf} and \eqref{eqoptcontrol2}; in \eqref{eqprob_n_inf} the logarithmic term is a fixed constant ($Q^*$ is given), while in \eqref{eqoptcontrol2} the logarithmic term depends on the chosen policy $Q$.
To solve  \eqref{eqoptcontrol2} by backward dynamic programming, for each $t\in\mathcal{T}$, introduce the value function:
\[
V_{t}(P_{t})  \triangleq 
\min_{\{Q_{\tau}\}_{\tau=t}^{T-1}} \sum_{\tau=t}^{T-1} \sum_{i,j} P_{\tau}^i Q_{\tau}^{ij}\!\left(C_\tau^{ij}+\alpha\log\frac{Q_\tau^{ij}}{R_\tau^{ij}} \right)
\]
and the associated Bellman equation
\begin{equation}
{\upd
V_{t}(P_{t}) =\min_{Q_t} \Big\{ \sum_{i,j} P_t^i Q_t^{ij}\!\left(C_t^{ij}+\alpha\log\frac{Q_t^{ij}}{R_t^{ij}} \right)+V_{t+1}(P_{t+1}) \Big\}} \label{eqbellmankl}
\end{equation}
with the terminal condition $V_T(\cdot)=0$.
The next theorem states that the Bellman equation \eqref{eqbellmankl} can be linearized by a change of variables (the Cole-Hopf transformation), and thus the
optimal control problem \eqref{eqoptcontrol2} is reduced to solving a linear system \cite{todorov2007linearly}.
\begin{theorem}
\label{theo1}
Let $\{\phi_t\}_{t\in\mathcal{T}}$ be the sequence of $V$-dimensional vectors defined by the backward recursion 
\begin{equation}
\label{eqphi}
\phi_t^i=\sum_j R_t^{ij} \exp \left(-\frac{C_t^{ij}}{\alpha}\right)\phi_{t+1}^j \;\; \forall i\in\mathcal{V}
\end{equation}
with the terminal condition $\phi_T^i=1 \; \forall i$.
Then, for each $t=0, 1, \cdots, T$ and $P_t$, the value function can be written as
\begin{equation}
V_t(P_t)=-\alpha\sum_i P_t^i \log \phi_t^i. \label{eqvt}
\end{equation}
Moreover, the optimal policy for \eqref{eqoptcontrol2} is given by
\begin{equation}
Q_t^{ij*}=\frac{\phi_{t+1}^j}{\phi_t^i}R_t^{ij}\exp\left(-\frac{C_t^{ij}}{\alpha}\right). \label{eqoptq}
\end{equation}
\end{theorem}
\begin{proof}
\ifdefined\LONGVERSION
Appendix~\ref{apptheo1}.
\else
\cite[Appendix C]{arXivversion}.
\fi
\end{proof}

We stress that \eqref{eqphi} is linear in $\phi$ and can be computed by matrix multiplications backward in time.

\section{Mean Field Equilibrium}
\label{secmfe}
{\upd
In this section, we investigate the relationship between the optimal control problem \eqref{eqoptcontrol2} and the fixed point condition \eqref{eqconsistency} for an MFE in the traffic routing game. 
To this end, we introduce the value function for the optimal control problem \eqref{eqprob_n_inf}, defined by
\[
\tilde{V}_{t}(P_{t})  \triangleq \min_{\{Q_{\tau}\}_{\tau=t}^{T-1}} \!\sum_{\tau=t}^{T-1} \sum_{i,j} P_{\tau}^i Q_{\tau}^{ij}\!\left(C_\tau^{ij}\!+\!\alpha\log\frac{Q_\tau^{ij*}}{R_\tau^{ij}} \right)
\]
The value function satisfies the Bellman equation:
\begin{equation}
\tilde{V}_{t}(P_{t})= \min_{Q_{t}} \Big\{ \sum_{i,j}P_{t}^i Q_{t}^{ij}\left(C_t^{ij}\!+\alpha\log\frac{Q_t^{ij*}}{R_t^{ij}}\right)+\tilde{V}_{t+1}(P_{t+1}) \Big\} \label{eqbellman3}
\end{equation}
with the terminal condition $\tilde{V}_T(\cdot)=0$.
We emphasize the distinction between $\tilde{V}_{t}(\cdot)$ and $V_{t}(\cdot)$.
As in the previous section, $V_{t}(\cdot)$ is the value function associated with the KL control problem \eqref{eqoptcontrol2}, whereas $\tilde{V}_{t}(\cdot)$ is the value function associated with the optimal control problem \eqref{eqprob_n_inf}. 
Despite this difference, the next lemma shows an intimate connection between $V_{t}(\cdot)$ and $\tilde{V}_{t}(\cdot)$.
In particular, if the parameter $Q^*$ in \eqref{eqprob_n_inf}  is chosen to be the solution to the KL control problem \eqref{eqoptcontrol2}, then the objective function in \eqref{eqprob_n_inf} becomes a constant that does not depend on  the decision variable $\{Q_t\}_{t\in\mathcal{T}}$ (the \emph{equalizer property}\footnote{We note that the equalizer property (the term borrowed from \cite{grunwald2004game}) of the minimizers of free energy functions is well-known in statistical mechanics,  information theory, and robust Bayes estimation theory. } of the optimal KL control policy). 
Moreover, under this circumstance, the value function $\tilde{V}_{t}(\cdot)$  for \eqref{eqprob_n_inf} coincides with the value function  $V_{t}(\cdot)$ for the KL control problem \eqref{eqoptcontrol2}.
\begin{lemma}
\label{lembestresponse}
If $\{Q_t^*\}_{t\in\mathcal{T}}$ in \eqref{eqprob_n_inf} is fixed to be the solution to the KL control problem \eqref{eqoptcontrol2}, then an arbitrary policy $\{Q_{t}\}_{t\in\mathcal{T}}$ with $Q_{t}\in\mathcal{Q}$ is an optimal solution to \eqref{eqprob_n_inf}.
Moreover, for each $t\in\mathcal{T}$ and $P_{t}$, we have 
\begin{equation}
\label{eqlembestresponse}
\tilde{V}_{t}(P_{t})=-\alpha \sum\nolimits_i P_{t}^i \log \phi_t^{i}
\end{equation}
where $\{\phi_t\}_{t\in\mathcal{T}}$ is the sequence calculated by \eqref{eqphi}.
\end{lemma}
\begin{proof}
We show \eqref{eqlembestresponse} by backward induction. If $t=T$, the claim trivially holds due to the definition $\tilde{V}_{T}(P_T)=0$ and the fact that the terminal condition for \eqref{eqphi} is given by $\phi_T^{i}=1$. Thus, for $0\leq t \leq T-1$, assume that
\[
\tilde{V}_{t+1}(P_{t+1})=-\alpha \sum_j P_{t+1}^j \log \phi_{t+1}^{j} 
\]
holds.
Using $\rho_t^{ij}=C_t^{ij}-\alpha \log \phi_{t+1}^{j}$, the Bellman equation \eqref{eqbellman3} can be written as 
\begin{equation}
\tilde{V}_{t}(P_{t})=\min_{Q_{t}} \sum_{i,j}P_{t}^i Q_{t}^{ij}\left(\rho_t^{ij}+\alpha\log\frac{Q_t^{ij*}}{R_t^{ij}}\right). \label{eqbellman4}
\end{equation}
Substituting $Q_t^{ij*}$ obtained by \eqref{eqoptq} into \eqref{eqbellman4}, we have
\begin{subequations}
\label{eqvnt}
\begin{align}
\tilde{V}_{t}(P_{t}) &=\min_{Q_{t}}\sum\nolimits_{i,j} P_{t}^i Q_{t}^{ij}\left(-\alpha \log \phi_t^{i}\right) \label{eqvntb1} \\
&=\min_{Q_{t}} \sum\nolimits_i P_{t}^i\left(-\alpha \log \phi_t^{i}\right) \underbrace{\sum\nolimits_j Q_{t}^{ij}}_{=1} \label{eqvntb2} \\
&=-\alpha \sum\nolimits_i P_{t}^i \log \phi_t^{i}. \label{eqvntb3}
\end{align}
\end{subequations}
This completes the proof of \eqref{eqlembestresponse}.
The chain of equalities \eqref{eqvnt} also shows that  the decision variable $Q_t$ vanishes in the ``min'' operator, indicating that any $Q_t\in \mathcal{Q}$ is a minimizer. 
This shows the equalizer property of $\{Q_t^*\}_{t\in\mathcal{T}}$.
\end{proof}
Lemma~\ref{lembestresponse} provides the following insights into the MFE of the traffic routing game: Suppose that all the players except the player $n$ adopt the policy $Q^*$ (the optimal solution to \eqref{eqoptcontrol2}) and the number of players tends to infinity. Since $Q^*$ will equalize the costs of all alternative routing policies for player $n$, any routing policy will be a best response for her. In particular, this means that the policy $Q^*$ itself will also be one of the best responses, and thus the fixed point condition \eqref{eqconsistency} will be satisfied.
Therefore, $Q^*$ will be an MFE of the considered traffic routing game. The following theorem, which is the main result of this paper, confirms this intuition.
\begin{theorem}
\label{theomfe}
A symmetric strategy profile $Q_{n,t}^{ij}=Q_t^{ij*}$ for each $n\in\mathcal{N}, t\in \mathcal{T}$ and $i,j\in\mathcal{V}$, where $Q_t^{ij*}$ is obtained by \eqref{eqphi}--\eqref{eqoptq}, is an MFE of the traffic routing game.
\end{theorem}
\begin{proof}
\ifdefined\LONGVERSION
Appendix~\ref{apptheomfe}.
\else
\cite[Appendix D]{arXivversion}.
\fi
\end{proof}
Theorem~\ref{theomfe}, together with Theorem~\ref{theo1}, provides an efficient algorithm for computing an MFE of the traffic routing game presented in Section~\ref{secprob}.
In particular, we remark that the MFE can be computed by the backward-in-time recursion  \eqref{eqphi}--\eqref{eqoptq}.
This is in stark contrast to the standard MFG formalism in which  a coupled pair of forward and backward equations must be solved to obtain an MFE.

Finally, we remark that the equalizer property of the MFE $Q^*$ characterized 
by Lemma~\ref{lembestresponse} is a reminiscent of the \emph{Wardrop's first principle}, stating that costs are equal on all the routes used at the  equilibrium. Although the costs usually mean journey times in the literature around Wardrop's principles \cite{wardrop1952some,correa2011wardrop}, the cost in our setting is the sum of the travel costs and the tax costs as stated in \eqref{eqobjn}. In this sense, Lemma~\ref{lembestresponse} can be viewed as an extension of the standard description of the Wardrop's first principle.
}

\section{Weak and strong time consistency}
\label{sectc}

{\upd This short section presents another notable property of the MFE derived in the previous section.}
Let $Q_{n,t}=Q_t$ for each $n\in\mathcal{N}$ and $0 \leq t \leq T-1$ be a symmetric strategy profile, and $P_t$ be the probability distribution over $\mathcal{V}$ induced by $Q_t$ as in \eqref{eqpq2}.
For every time step $0 \leq t \leq T-1$, a dynamic game restricted to the time horizon $\{t, t+1, ... , T-1\}$ with the initial condition $P_t$ is called the \emph{subgame} of the original game.
The following are natural extensions of the \emph{strong and weak time consistency} concepts  in the dynamic game theory \cite{basar1999dynamic} to MFGs.
\begin{definition}
An MFE strategy profile $Q^*$ is said to be:
\begin{enumerate}[leftmargin=*]
\item \emph{weakly time-consistent} if for every $0 \leq t \leq T-1$, $\{Q_s^*\}_{t\leq s \leq T-1}$ constitutes an MFE of the subgame restricted to $\{t, t+1, ... , T-1\}$ when $\{Q_s\}_{0\leq s \leq t-1}=\{Q_s^*\}_{0\leq s \leq t-1}$.
\item \emph{strongly time-consistent} of for every $0 \leq t \leq T-1$, $\{Q_s^*\}_{t\leq s \leq T-1}$ constitutes an MFE of the subgame restricted to $\{t, t+1, ... , T-1\}$ when regardless of the policy $\{Q_s\}_{0\leq s \leq t-1}$ implemented in the past.
\end{enumerate}
\end{definition}
In the standard MFG formalism \cite{caines2013mean,lasry2007mean} where the MFE is characterized by a forward-backward HJB-FPK system, the equilibrium policy is only  weakly time-consistent in general. This is because, in the event of $P_t$ not being consistent with the distribution induced by $\{Q_s^*\}_{0\leq s \leq t-1}$, the MFE of the subgame must be recalculated by solving the HJB-FPK system over $t\leq s \leq T-1$. 
In contrast, the MFE considered in this paper is characterized only by a backward equation (Theorems~\ref{theo1} and \ref{theomfe}). A notable consequence of this fact is that even if the initial condition $P_t$ is inconsistent with the planned distribution, it does not alter the fact that $\{Q_s^*\}_{t\leq s \leq T-1}$ constitutes an MFE of the subgame restricted to $t\leq s \leq T-1$. Therefore, the MFE characterized by Theorems~\ref{theo1} and \ref{theomfe} is strongly time-consistent.

\section{Mean field equilibrium and fictitious play}
\label{secfictitious}

{\upd The equalizer property of the MFE characterized by Lemma~\ref{lembestresponse} raises the following question regarding the stability of the equilibrium: 
If the MFE equalizes the costs of all the available route selection policies, what incentivizes individual players to stay at the MFE policy $Q^*$? 
In this section, we reason about the stability of MFE by relating it with the convergence of the \emph{fictitious play} process \cite{monderer1996fictitious} for an associated repetitive game.
Convergence of fictitious play processes have been studied in depth in \cite{monderer1996fictitious} and \cite{shamma2005dynamic}. We also remark that fictitious play for day-to-day policy adjustments for traffic routing has been considered in \cite{garcia2000fictitious,xiao2013average}.
Fictitious play in the context of MFGs has been studied in the recent work \cite{cardaliaguet2017learning}.
}

\begin{figure}[t]
    \centering
    \includegraphics[width=0.6\columnwidth]{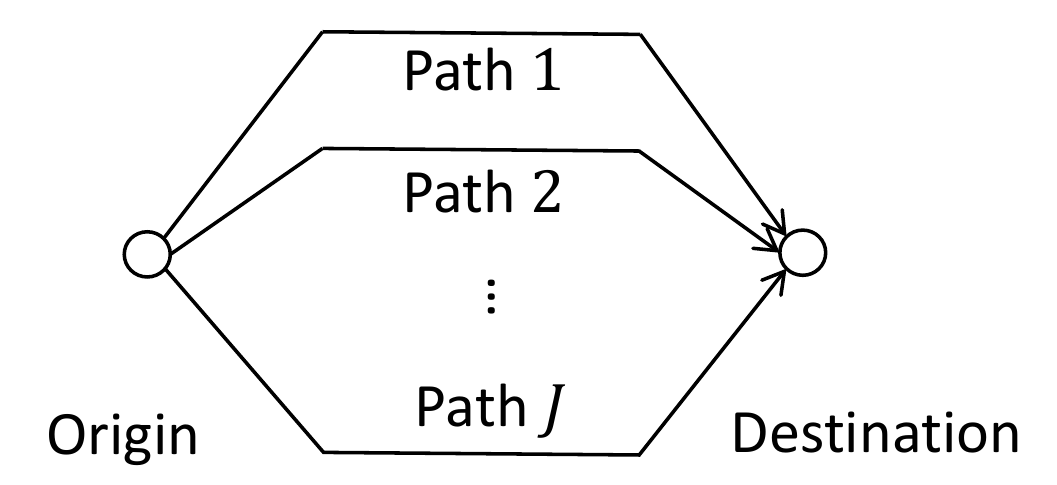}
    \caption{Simple path choice problem.}
    \label{fig:path}
    \vspace{-5ex}
\end{figure}

Consider the situation in which the traffic routing game is repeated on a daily basis, and individual players update their routing policies based on their past experiences. 
For simplicity, we only consider a single-origin-single-destination, $N$-player traffic routing game shown in Figure~\ref{fig:path}. We assume that there are $J$ parallel routes from the origin to the destination. All players are initially located at the origin node. 
Each route $j$ is associated with the travel cost $c^j$ and the tax cost $\alpha\log\frac{K_N^j}{NR^j}$, where $K_N^j$ is the number of players selecting route $j$. As before, $\alpha, C^j, R^j$ are given constants. {\upd 
By \emph{fictitious play}, we mean the following day-to-day policy adjustment mechanism for individual players:
On day one, each player $n\in\mathcal{N}$ makes initial guesses on player $m$'s mixed strategies (for all $m\neq n$) for their route selection. Player $n$'s belief on player $m$'s policy is demoted by $Q_{n\rightarrow m}[1]\in\Delta^{J-1}$. Assuming that $Q_{n\rightarrow m}[1], \forall m\neq n$ are fixed, player $n$ selects a route with the lowest expected cost. Player $n$'s route selection is observed and recorded by all players at the end of day one.
On day $\ell$, player $n$'s belief $Q_{n\rightarrow m}[\ell]\in\Delta^{J-1}$ for each $m\neq n$ is set to be equal to the vector of observed empirical frequencies of player $m$'s route choices up to day $\ell-1$. The process is repeated on a daily basis. We call $Q_{n\rightarrow m}[\ell]$, $\ell=1,2, \cdots$ for all pairs $m\neq n$ the \emph{belief paths}. 
}
The process is summarized in Algorithm~\ref{alg1}.

\begin{algorithm}[h]
  \label{alg1}
  \caption{The fictitious play process for the simplified traffic routing game.}

{\bf Step 0:}  On day one, each player $n$ initializes a mixed strategy in belief
$
Q_{n\rightarrow m}[1] \in \Delta^{J-1}$ for each $m \neq n
$
according to which she believes player $m$ select routes. 

{\bf Step 1:} At the beginning of day $\ell$, each player $n$ fixes assumed mixed strategy $Q_{n\rightarrow m}[\ell]\in \Delta^{J-1}$ according to which she believes player $m$ select routes. Based on this assumption, she selects her best response $r_n[\ell]=\argmin_j y_n^j[\ell]$, where $y_n^j[\ell]$ is the assumed cost of selecting route $j$, i.e.,
\begin{equation}
\label{eqfppbd}
y_{n}^j[\ell]=\mathbb{E} \left(C^j + \alpha \log \frac{K_N^j}{NR^j}\right).
\end{equation}

{\bf Step 2:}  At the end of day $k$, each player $n$ updates her belief based on observations $r_m[\ell], m \neq n$ by
\begin{equation}
\label{eqfpupdate}
Q_{n\rightarrow m}[\ell+1]=\frac{\ell}{\ell+1}Q_{n\rightarrow m}[\ell]+\frac{1}{\ell+1}\delta(r_m[\ell])
\end{equation}
where $\delta(r)$ is the indicator vector whose $r$-th entry is 
one and all other entries are zero. 
Return to Step 1.
\end{algorithm}

{\upd In what follows, we show that Algorithm~\ref{alg1} converges to a unique symmetric Nash equilibrium of the $N$-player game shown in Figure~\ref{fig:path} if the initial belief is symmetric. A numerical simulation is presented  in Section~\ref{secfpsim} to demonstrate this convergence behavior, where we also observe that the policy obtained in the limit of the belief path is closely approximated by the MFE if $N$ is sufficiently large. This observation provides the MFE with an interpretation as a steady-state value of the players' day-to-day belief adjustment processes in a large population traffic routing game.

Convergence of Algorithm~\ref{alg1} is a straightforward consequence of Monderer and Shapley \cite{monderer1996fictitious}, where it is shown that every belief path for $N$-player games with identical payoff functions converges to an equilibrium. 
This result is directly applicable to the $N$-player traffic routing game shown in Figure~\ref{fig:path} since it is clearly a symmetric game. The only caveat is that there is no guarantee that the belief path converges to a symmetric equilibrium if the game has multiple Nash equilibria (including non-symmetric ones). However, this difficulty can be circumvented if we impose an additional assumption that the initial belief is symmetric, i.e., $Q_{n\rightarrow m}[1]=Q[1]$ for some $Q[1]\in\Delta^{J-1}$ for all $(m,n)$ pairs. If the initial belief is symmetric, the belief path generated by Algorithm~\ref{alg1} remains symmetric, i.e., $Q_{n\rightarrow m}[\ell]=Q[\ell]$, $y_n[\ell]=y[\ell]$ and $r_n[\ell]=r[\ell]$ for $\ell\geq 1$. In this case, equations \eqref{eqfppbd} and \eqref{eqfpupdate} are simplified to
\begin{align*}
y^i[\ell]=C^j+\alpha \sum_{k=0}^{N-1} \log &\left(\frac{k+1}{NR^j}\right)
 {N\!-\!1 \choose k}\\&\times (Q^j[\ell])^k (1-Q^j[\ell])^{N-1-k} 
\end{align*}
and
\[
Q[\ell+1]=\frac{\ell}{\ell+1}Q[\ell]+\frac{1}{\ell+1}\delta(r[\ell])
\]
respectively. 
Combined with the convergence result by Monderer and Shapley \cite{monderer1996fictitious}, it can be concluded that every limit point of the belief path generated by Algorithm~\ref{alg1} with symmetric initial belief is a symmetric equilibrium.}

The next lemma shows that there exists a unique symmetric equilibrium in the simple traffic routing game in Figure~\ref{fig:path} with finite number of players.
\begin{lemma}
\label{lemsinglestage}
There exists a unique symmetric equilibrium, denoted by $Q^{(N)*}$, in the $N$-player traffic routing game shown in Figure~\ref{fig:path}.
\end{lemma}
\begin{proof}
\ifdefined\LONGVERSION
Appendix~\ref{app2}.
\else
\cite[Appendix E]{arXivversion}.
\fi
\end{proof}
{\upd
Now, consider the limit $\lim_{N\rightarrow \infty} Q^{(N)*}$ and its relationship with the MFE $Q^*$.
Notice that the MFE of the traffic routing game in Figure~\ref{fig:path} is characterized as the unique solution to the following convex optimization problem:}
\begin{equation}
\label{eqfpprob}
\min_{Q\in\Delta^{J-1}}  \sum_{j=1}^J Q^j \left(C^j+\alpha\log\frac{Q^j}{R^j}\right).
\end{equation}
{\upd
In Section~\ref{secfpsim}, we perform a simulation study where we observe that $Q^{*}$ is a good approximation of $Q^{(N)*}$ when $N$ is sufficiently large. Although the condition under which the identity $\lim_{N\rightarrow \infty} Q^{(N)*}=Q^{*}$ holds must be studied carefully in the future,\footnote{While Lemma~\ref{lemsinglestage} establishes the uniqueness of the symmetric equilibrium for a simple traffic routing game shown in  Figure~\ref{fig:path}, its extension to the general class of traffic routing game formulated in Section~\ref{secprob} is currently unknown. The proof of the identity $\lim_{N\rightarrow \infty} Q^{(N)*}=Q^{*}$ (for both the simple game in Figure~\ref{fig:path} and the general setup in Section~\ref{secprob}) must be postponed as future work.} this observation suggests an important practical interpretation of the MFE: namely, it is an approximation of the limit point of the belief path (or equivalently, the empirical frequency of each player to take particular routes) of the symmetric fictitious play when $N$ is large. This provides an answer to the question regarding the stability of the MFE raised in the beginning of this section.}

\section{Numerical Illustration}
\label{secsimulation}
{\upd 
In this section, we present numerical simulations that illustrate the main results obtained in Sections~\ref{secmfe} and \ref{secfictitious}.}

\subsection{Traffic routing game and congestion control}
\label{secsimulation1}
We first illustrate the result of Theorem~\ref{theomfe} applied to a traffic routing game shown in Fig.~\ref{fig:simulation}.
At $t=0$, the population is concentrated in the origin cell (indicated by ``O''). 
For $t\in\mathcal{T}$, the travel cost for each player is 
\[
C_t^{ij}=\begin{cases}
C_{\text{term}} & \text{ if } j=i\\
1+C_{\text{term}}  & \text{ if } j\in\mathcal{V}(i)\\
100000+C_{\text{term}}  & \text{  if } j\not\in\mathcal{V}(i) \text{ or } j \text{ is an obstacle}
\end{cases}
\]
where $\mathcal{V}(i)$ contains the north, east, south, and west neighborhood of the cell $i$. To incorporate the terminal cost, we introduce $C_{\text{term}}=0$ if $t=0, 1, \cdots, T-1$ and $C_{\text{term}}=10\sqrt{\text{dist}(j, \text{D})}$ if $t=T-1$, where $\text{dist}(j, \text{D})$ is the Manhattan distance between the player's final location $j$ and the destination cell (indicated by ``D''). 
As the reference distribution, we use $R_t^{ij}=1/|\mathcal{V}(i)|$ (uniform distribution) for each $i\in\mathcal{V}$ and $t\in \mathcal{T}$ to incentivize players to spread over the traffic graph.

For various values of $\alpha >0$, the backward formula \eqref{eqphi} is solved and the optimal policy is calculated by \eqref{eqoptq}.
If $\alpha$ is small (e.g., $\alpha=0.1$), it is expected that players will take the shortest path since the action cost is dominant compared to the tax cost \eqref{eqtax}. This is confirmed by numerical simulation; three figures in the top row of Fig.~\ref{fig:simulation} show snapshots of the population distribution at time steps $t=20, 35$ and $50$. In the bottom row, similar plots are generated with a larger $\alpha$ ($\alpha=1$). In this case, it can be seen that the equilibrium strategy will choose longer paths with higher probability to reduce congestion.

\begin{figure}[t]
    \centering
    \includegraphics[width=\columnwidth]{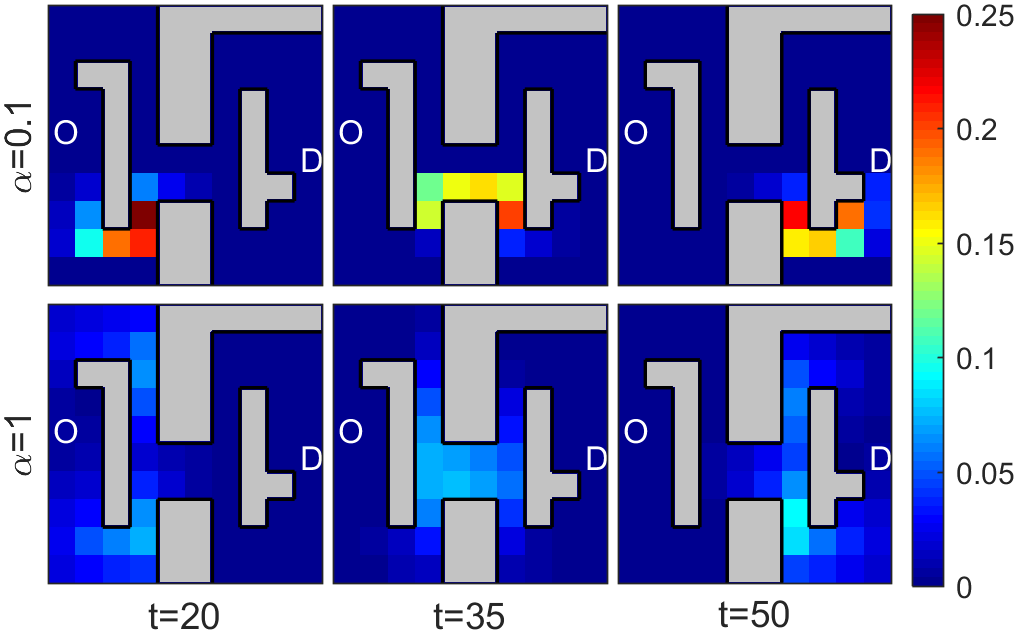}
    \caption{Mean-field traffic routing game with $T=70$ over a traffic graph with $100$ nodes (grid world with obstacles). Plots show vehicle distribution $P_t$ at $t=20, 35, 50$ and for $\alpha=0.1$ and $1$.}
    \label{fig:simulation}
    \vspace{-5ex}
\end{figure}

\subsection{Symmetric fictitious play}
\label{secfpsim}

Next, we present a numerical demonstration of the symmetric fictitious play studied in Section~\ref{secfictitious}. 
Consider a simple traffic graph in Figure~\ref{fig:path} with three alternative paths ($J=3$). We set travel costs $(C^1, C^2, C^3)=(2,1,3)$, while fixing $R^1=R^2=R^3=1/3$ and $\alpha=1$. Figure~\ref{fig:fp} shows the belief path generated by the policy update rule \eqref{eqfpupdate} with the initial policy $Q[1]=(1/3, 1/3, 1/3)$. The left shows the case with $20$ players ($N=20$), while the right plot shows the case with $N=200$. The MFE 
\[
Q^*=\frac{1}{\sum_{j} R^j\exp(-c^j)} \begin{bmatrix} R^1\exp(-c^1) \\
R^2\exp(-c^2) \\
R^3\exp(-c^3)
\end{bmatrix}=\begin{bmatrix} 0.245 \\
0.665 \\
0.090
\end{bmatrix}
\]
is also shown in each plot. The plot for $N=20$ shows that, while the belief path is convergent, there is a certain offset between its limit point and the MFE. This is because the number of players is not sufficiently large. On the other hand, when $N=200$, the MFE $Q^*$ is a good approximate to the limit point of the belief path. 

\begin{figure*}[t]
    \centering
    \includegraphics[width=0.8\textwidth]{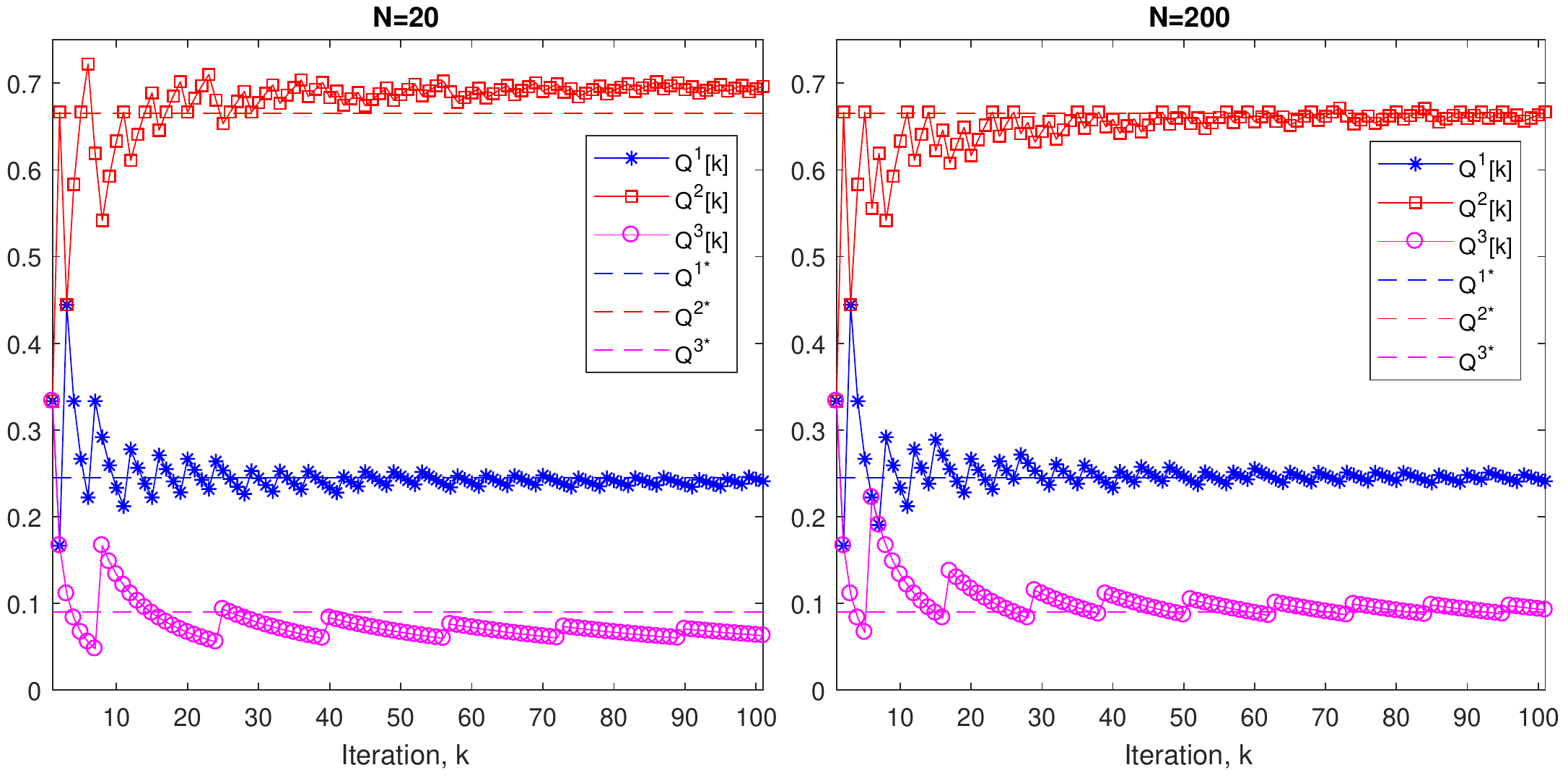}
    \caption{Convergence of the belief path generated by the symmetric fictitious play \eqref{eqfpupdate} in the $N$-player single-stage traffic routing game with three ($J=3$) alternative paths shown in Figure~\ref{fig:path}. The left plot shows the case with $N=20$ while the right plot show the case with $N=200$. The value of MFE $Q^*$ is also shown.}
    \label{fig:fp}
\end{figure*}

\section{Conclusion and Future Work}
\label{secconclude}
In this paper, we showed that the MFE of a large-population traffic routing game under the log-population tax mechanism can be obtained via the linearly solvable MDP.
Strong time consistency of the derived MFE was discussed. 
A connection between the MFE and fictitious play was investigated.

While this paper is restricted discrete-time discrete-state formalisms, its continuous-time continuous-state counterpart is worth investigating in the future. The interface between the existing traffic SUE theory \cite{sheffi1984urban} and MFG must be thoroughly studied in the future work. Convergence of fictitious play and its relationship with MFE presented in Section~\ref{secfictitious} should be studied in more general settings. {\upd Linear solvability renders the proposed MFG framework attractive as an incentive mechanism for TSOs for the purpose of traffic congestion mitigation; however, questions from the perspectives of mechanism design theory, such as how to tune parameters $\alpha$ and $R$ (which are assumed given in this paper) to balance the efficiency and budget, are unexplored.  
Finally, generalization to non-homogeneous MFGs with multiple classes of players (which was recently studied in \cite{pedram2019}) needs further investigation.
}

\appendices

\ifdefined\LONGVERSION
\section{Explicit expression of \eqref{eqexp_pi_general}}
\label{appbinomial}

Since player $n$'s probability of taking action $j$ at location $i$ at time step $t$ is given by $P_{n,t}^iQ_{n,t}^{ij}$, the total number $K_{N,t}^{ij}$ of such players follows the Poisson binomial distribution
\[
\text{Pr}(K_{N,t}^{ij}=k)=\sum_{A\in F_k}\prod_{n\in A} P_{n,t}^iQ_{n,t}^{ij}\!\!\prod_{n^c\in A^c}\!\!(1-P_{n^c,t}^iQ_{n^c, t}^{ij}).
\]
Here, $F_k$ is the set of all subsets of size $k$ that can be selected from $\mathcal{N}=\{1,2, ..., N\}$, and $A^c=F_k\backslash A$. Similarly, the distribution of $K_{N,t}^{i}$ is given by
\[
\text{Pr}(K_{N,t}^{i}=k)=\sum_{A\in F_k}\prod_{n\in A} P_{n,t}^i\!\!\prod_{n^c\in A^c}\!\!(1-P_{n^c,t}^i).
\]
Notice also that the conditional distribution of $K_{N,t}^{ij}$ given player $n$'s location-action pair $(i_{n,t},j_{n,t})=(i,j)$ is
\begin{align}
&\text{Pr}(K_{N,t}^{ij} =k+1 \mid i_{n,t}=i, j_{n,t}=j) \nonumber \\
&=\sum_{A\in F_k^{-n}}\prod_{m\in A} P_{m,t}^iQ_{m,t}^{ij}\!\!\prod_{m^c\in A^{-c}}\!\!(1-P_{m^c,t}^iQ_{m^c,t}^{ij}). \label{eqrhadcond}
\end{align}
Here, $F_k^{-n}$ is the set of all subsets of size $k$ that can be selected from $\mathcal{N}\backslash \{n\}$, and $A^{-c}=F_k^{-n}\backslash A$. Similarly, the conditional distribution of $K_{N,t}^{i}$ given $i_{n,t}=i$ is
\begin{align}
&\text{Pr}(K_{N,t}^{i} =k+1 \mid i_{n,t}=i) \nonumber \\
&=\sum_{A\in F_k^{-n}}\prod_{m\in A} P_{m,t}^i\!\prod_{m^c\in A^{-c}}\!\!(1-P_{m^c,t}^i). \label{eqrhadcond2}
\end{align}
Therefore, given the prior knowledge that the player $n$'s location-action pair at time $t$ is $(i,j)$, the expectation of her tax penalty $\pi_{N,n,t}^{ij}$  is 
\begin{align}
&\Pi_{N,n,t}^{ij}\triangleq \mathbb{E}\left[\pi_{N,n,t}^{ij} \mid  i_{n,t}=i, j_{n,t}=j \right] \nonumber \\
=&\sum_{k=0}^{N-1} \!\alpha\log\frac{k+1}{N}\!\!\!\!\sum_{A\in F_k^{-n}}\!\prod_{m\in A}\!\! P_{m,t}^iQ_{m,t}^{ij} \!\!\!\! \prod_{m^c\in A^c}\!\!\!(1\!-\!P_{m^c\!, t}^iQ_{m^c\!, t}^{ij}) \nonumber \\
&- \sum_{k=0}^{N-1} \!\alpha\log\frac{k+1}{N}\!\!\!\!\sum_{A\in F_k^{-n}}\!\prod_{m\in A}\!\! P_{m,t}^i \!\!\!\! \prod_{m^c\in A^c}\!\!\!(1\!-\!P_{m^c\!, t}^i) \nonumber \\
&- \alpha\log R_t^{ij}.
\label{eqtijgeneral}
\end{align}
Notice that the quantity \eqref{eqtijgeneral} depends on the strategies $Q_{-n}\triangleq \{Q_m\}_{m \neq n}$, but not on $Q_n$. In other words, $\pi_{N,n,t}^{ij}$ is a random variable whose distribution does not depend on player $n$'s own strategy. 

To evaluate $\Pi_{N,n,t}^{ij}$ when all players other than player $n$ takes the same strategy (i.e., $Q_{m,t}=Q_t^*$ for $m\neq n$), notice that the conditional distributions of $K_{N,t}^{ij}$ and $K_{N,t}^{i}$ given $(i_{n,t}, j_{n,t})=(i,j)$, provided by \eqref{eqrhadcond} and \eqref{eqrhadcond2}, simplify to the binomial distributions
\begin{align*}
&\text{Pr}(K^{ij}_{N,t}=k+1|i_{n,t}=i, j_{n,t}=j)\\
&={N\!-\!1 \choose k}(P_t^{i*}Q_t^{ij*})^k (1-P_t^{i*}Q_t^{ij*})^{N-1-k}\\
&\text{Pr}(K^{i}_{N,t}=k+1|i_{n,t}=i)\\
&={N\!-\!1 \choose k}(P_t^{i*})^k (1-P_t^{i*})^{N-1-k}.
\end{align*}
Thus, the expression \eqref{eqtijgeneral} simplifies to
\begin{align}
&\Pi_{N,n,t}^{ij}=\mathbb{E}\left[\pi_{N,n,t}^{ij} \mid  i_n=i, j_n=j \right] \nonumber \\
&=\!\sum_{k=0}^{N-1}\!\alpha \log\frac{k+1}{N}{N\!-\!1 \choose k}(P_t^{i*}Q_t^{ij*})^k (1\!-\!P_t^{i*}Q_t^{ij*})^{N-1-k} \nonumber \\
&\;\;\;-\!\sum_{k=0}^{N-1}\!\alpha \log\frac{k+1}{N}{N\!-\!1 \choose k}(P_t^{i*})^k (1\!-\!P_t^{i*})^{N-1-k} \nonumber \\
&\;\;\;- \alpha \log R_t^{ij}
 \label{eqtjieq}
\end{align}
\fi

\ifdefined\LONGVERSION
\section{Proof of Lemma~\ref{lemlimpi}}
\label{app1}

Let $K^i_{N,-n,t}$ denote the number of agents, except agent $n$, which are located at intersection $i$ at time $t$ and let $K^{ij}_{N,-n,t}$ denote the number of agents, except agent $n$, which are located at intersection $i$ at time  $t$ and select intersection $j$ as their next destination. Thus, we have  $K^i_{N,-n,t}=\sum_{l\neq n}\I{i_{l,t}=i}$ and $K^{ij}_{N,-n,t}=\sum_{l\neq n}\I{i_{l,t}=i, j_{l,t}=j}$,
where $\I{\cdot}$ is the indicator function.
Then, $\Pi_{N,n,t}^{ij*}$ can be written as 
\begin{align}
\Pi_{N,n,t}^{ij*}=&\ES{\log\tfrac{1+K^{ij}_{N,-n,t}}{1+K^i_{N,-n,t}}}-\log R_t^{ij}\nonumber\\
=&\ES{\logp{ \tfrac{1+K^{ij}_{N,-n,t}}{N}}}-\nonumber\\
&\ES{\logp{\tfrac{1+K^i_{N,-n,t}}{N}}}-\log R_t^{ij}\nonumber
\end{align}
Using Jensen inequality, we have 
\begin{align}
 \ES{\logp{\tfrac{1+K^{ij}_{N,-n,t}}{N}}}&\leq \logp{\frac{1}{N}+\ES{\tfrac{K^{ij}_{N,-n,t}}{N}}}\nonumber\\
 &\stackrel{(a)}{=}\logp{\frac{1}{N}+\frac{N-1}{N}P_t^{i*}Q_t^{ij*}}\nonumber
\end{align}
where $(a)$ follows from $\ES{\frac{K^{ij}_{N,-n,t}}{N}}=\frac{N-1}{N}P_t^{i*}Q_t^{ij*}$ and the fact that all the agents employ the policy $\left\{Q^{*}_t\right\}$. Thus, we have 
\begin{align}
\limsup_{N\rightarrow\infty}\ES{\log \tfrac{1+K^{ij}_{N,-n,t}}{N}}\leq \log P_t^{i*}Q_t^{ij*}
\end{align}
Next we show the other direction. For $\epsilon\in\left(0,\right.\left.\frac{ P_t^{i*}Q_t^{ij*}}{2}\right]$, we can write $\ES{\log \frac{1+K^{ij}_{N,-n,t}}{N}}$ as
\begin{align}
\ES{\log \tfrac{1+K^{ij}_{N,-n,t}}{N}}=&\ES{\log \tfrac{1+K^{ij}_{N,-n,t}}{N}\I{\tfrac{K^{ij}_{N,-n,t}}{N}>\epsilon}}+\nonumber\\
&\ES{\log \tfrac{1+K^{ij}_{N,-n,t}}{N}\I{\tfrac{K^{ij}_{N,-n,t}}{N}\leq \epsilon}}\nonumber
\end{align}

Using the Hoeffding inequality, it follows that $\frac{K^{ij}_{N,-n,t}}{N}$ converges to $P_t^{i*}Q_t^{ij*}$ in probability as $N$ becomes large. From continues mapping theorem, we have the convergence  of $\log \frac{K^{ij}_{N,-n,t}}{N}$ in probability to $\log P_t^{i*}Q_t^{ij*} $ for $P_t^{i*}Q_t^{ij*}>0$. Similarly, $\I{\frac{K^{ij}_{N,-n,t}}{N}>\epsilon}$ converges to 1 in probability. Thus, from Slutsky's Theorem, we have $\log \frac{1+K^{ij}_{N,-n,t}}{N}\I{\frac{K^{ij}_{N,-n,t}}{N}>\epsilon}$ converges to $\log P_t^{i*}Q_t^{ij*}$ in distribution. Using Fatou's lemma and the fact that $\log \frac{1+K^{ij}_{N,-n,t}}{N}\I{\frac{K^{ij}_{N,-n,t}}{N}>\epsilon}\geq \log\epsilon $ , we have 
\begin{align}
\liminf_{N\rightarrow\infty} \ES{\log\tfrac{1+K^{ij}_{N,-n,t}}{N}\I{\tfrac{K^{ij}_{N,-n,t}}{N}>\epsilon}}\geq \log P_t^{i*}Q_t^{ij*}\nonumber
\end{align}
We also have 
\begin{align*}
&\abs{\ES{\log \tfrac{1+K^{ij}_{N,-n,t}}{N}\I{\tfrac{K^{ij}_{N,-n,t}}{N}\leq \epsilon}}} \\
&\leq \logp{N}\PRP{\tfrac{K^{ij}_{N,-n,t}}{N}\leq \epsilon}
\end{align*}
Using the Hoeffding inequality, it is straightforward to show that $\PRP{\frac{K^{ij}_{N,-n,t}}{N}\leq \epsilon}$ decays to zero exponentially in $N$ which implies that 
\begin{align}
\lim_{N\rightarrow\infty}\ES{\log \tfrac{1+K^{ij}_{N,-n,t}}{N}\I{\tfrac{K^{ij}_{N,-n,t}}{N}\leq \epsilon}}=0\nonumber
\end{align}
Thus, we have 
\begin{align}
\liminf_{N\rightarrow\infty}\ES{\log \tfrac{1+K^{ij}_{N,-n,t}}{N}}\geq \log P_t^{i*}Q_t^{ij*}
\end{align}
which implies that $\lim_{N\rightarrow\infty}\ES{\logp{ \frac{1+K^{ij}_{N,-n,t}}{N}}}=\log P^{i*}_tQ^{ij*}_t$. Following similar steps, it is straightforward to show that $\lim_{N\rightarrow\infty}\ES{\logp{ \frac{1+K^{i}_{N,-n,t}}{N}}}=\log P^{i*}_t$ which completes the proof.
\fi

\ifdefined\LONGVERSION
\section{Proof of Theorem~\ref{theo1}}
\label{apptheo1}
Due to the choice of the terminal condition $\phi_T^i=1$,
\[
V_T(P_T)=\sum_iP_T^i C_T^i = -\alpha\sum_i P_T^i \log \phi_T^i=0
\]
holds. To complete the proof by backward induction, assume that
$
V_{t+1}(P_{t+1})=-\alpha\sum_i P_{t+1}^i \log \phi_{t+1}^i
$
holds for some $0\leq t \leq T-1$. {\upd Then, due to the Bellman equation \eqref{eqbellmankl}, we have 
\begin{equation}
\label{eq_apdx_c}
 V_{t}(P_{t}) =\min_{Q_t} \left\{ \sum_{i,j} P_t^i Q_t^{ij}\left(\rho_t^{ij}+\alpha\log\frac{Q_t^{ij}}{R_t^{ij}} \right)\right\}
\end{equation}
where $\rho_t^{ij}=C_t^{ij}-\alpha\log \phi_{t+1}^j$ is a constant.
To obtain a minimizer for \eqref{eq_apdx_c} subject to the constraints $P_t^i\left(\sum_j Q_t^{ij}-1\right)=0 \; \forall i$, introduce the Lagrangian function
\begin{align*}
L(Q_t, \lambda_t)=& \sum_{i,j} P_t^i Q_t^{ij}\left(\rho_t^{ij}+\alpha\log\frac{Q_t^{ij}}{R_t^{ij}} \right) \\
&-\sum_i \lambda_t^i P_t^i \left(\sum_j Q_t^{ij}-1\right).
\end{align*}
For each $i$ such that $P_t^i>0$, we obtain from the optimality condition $\frac{\partial L}{\partial Q_t^{ij}}=0$ that
\[
Q_t^{ij}=R_t^{ij}\exp\left(\frac{\lambda_t^i}{\alpha}-1-\frac{\rho_t^{ij}}{\alpha}\right).
\]
In order to satisfy $\sum_j Q_t^{ij}=1$, the Lagrange multiplier $\lambda_t^i$ must be chosen so that
\[
Q_t^{ij}=\frac{R_t^{ij}}{Z_t^i}\exp\left(-\frac{\rho_t^{ij}}{\alpha}\right)
\]
where the normalization constant $Z_t^i$ can be computed as
\begin{align*}
Z_t^i&=\sum_j R_t^{ij}\exp\left(-\frac{\rho_t^{ij}}{\alpha}\right) \\
&=\sum_j R_t^{ij}\exp\left(-\frac{C_t^{ij}}{\alpha}\right)\phi_{t+1}^j \\
&=\phi_t^i.
\end{align*}
Therefore, we have shown that a minimizer for \eqref{eq_apdx_c} is obtained as
\[
Q_t^{ij*}=\frac{R_t^{ij}}{\phi_t^i}\exp\left(-\frac{\rho_t^{ij}}{\alpha}\right)
\]
from which \eqref{eqoptq} also follows.} By substitution, the optimal value is shown to be
$
V_t(P_t)=-\alpha \sum_i P_t^i \log \phi_t^i.
$
This completes the induction proof.
\fi

\ifdefined\LONGVERSION
\section{Proof of Theorem~\ref{theomfe}}
\label{apptheomfe}
Let the policies $Q_{m,t}^{ij}=Q_t^{ij*}$ for $m \neq n$ be fixed. It is sufficient to show that there exists a sequence $\epsilon_N \searrow 0$ such that the cost of adopting a strategy $Q_{n,t}^{ij}=Q_t^{ij*}$  for player $n$ is no greater than $\epsilon_N $ plus the cost of adopting any other policy.
Since 
\[
\Pi_{N,n,t}^{ij}\rightarrow \alpha \log \frac{Q_t^{ij*}}{R_t^{ij}} \text{ as } N\rightarrow \infty,
\]
there exists a sequence $\delta_N \searrow 0$ such that
\[
\Pi_{N,n,t}^{ij}+\delta_N > \alpha \log \frac{Q_t^{ij*}}{R_t^{ij}} \;\; \forall i, j, t.
\]
Now, for all policy $\{Q_{n,t}\}_{t\in\mathcal{T}}$ of player $n$ and the induced distributions $P_{n,t}^j=\sum_i P_{n,t}^i Q_{n,t}^{ij}$, we have
\begin{align*}
&\sum_{t=0}^{T-1}\sum_{i,j} P_{n,t}^i Q_{n,t}^{ij} \left(C_t^{ij}+\Pi_{N,n,t}^{ij}\right) \\
&>\sum_{t=0}^{T-1}\sum_{i,j} P_{n,t}^i Q_{n,t}^{ij} \left(C_t^{ij}+\alpha \log \frac{Q_t^{ij*}}{R_t^{ij}}- \delta_N \right) \\
&=\sum_{t=0}^{T-1}\sum_{i,j} P_{n,t}^i Q_{n,t}^{ij} \left(C_t^{ij}+\alpha \log \frac{Q_t^{ij*}}{R_t^{ij}}\right)- T V^2 \delta_N\\
&\geq \min_{\{Q_{n,t}\}_{t\in\mathcal{T}}}\sum_{t=0}^{T-1}\sum_{i,j} P_{n,t}^i Q_{n,t}^{ij} \left(C_t^{ij}+\alpha \log \frac{Q_t^{ij*}}{R_t^{ij}}\right) \\
& \hspace{5ex} - T V^2 \delta_N. 
\end{align*}
Notice that the minimization in the last line is attained by adopting $Q_{n,t}^{ij}=Q_t^{ij*}$. Since $\epsilon_N \triangleq T V^2 \delta_N \searrow 0$, this completes the proof.
\fi

\ifdefined\LONGVERSION
\section{Proof of Lemma~\ref{lemsinglestage}}
\label{app2}

For each $j=1, 2, ... , J$, define $f_N^j: [0, 1]\rightarrow \mathbb{R}$ by
\begin{align*}
&f_N^j(Q^j)\triangleq \\
& c^j+\alpha\sum_{n=0}^{N-1}\log\frac{n+1}{NR^j}{N\!-\!1 \choose n}(Q^j)^n (1-Q^j)^{N-1-n}.
\end{align*}
Notice that $f_N^j(Q)$ is a continuous and strictly increasing function. If $Q^*\in \Delta^{J-1}$ is a symmetric Nash equilibrium of the $N$-player single-stage traffic routing game, it satisfies the condition
\begin{equation}
\label{eqNEcond}
Q^{*}\in \argmin_{Q \in \Delta^{J-1}} \sum_{j=1}^J Q^j f_N^j(Q^{j*}).
\end{equation}
From the KKT condition, $Q^*\in \Delta^{J-1}$ satisfies \eqref{eqNEcond} if and only if there exists $\lambda\in\mathbb{R}$ such that
\begin{subequations}
\label{eqKKT}
\begin{align}
f_N^j(Q^{j*})&=\lambda \text{ for all } j \text{ such that } Q^{j*}>0 \label{eqKKT1} \\
f_N^j(0)&\geq \lambda \text{ for all } j \text{ such that } Q^{j*}=0. \label{eqKKT2}
\end{align}
\end{subequations}
Alternatively, the above condition can directly be obtained from the Wardrop's first principle \cite{wardrop1952some}.
Thus, it is sufficient to show that there exists a unique $Q^*\in \Delta^{J-1}$ satisfying the condition \eqref{eqKKT}. For each $j=1, 2, ... , N$, define $g^j:\mathbb{R}\rightarrow [0,1]$ by
\begin{equation}
\label{eqgdef}
g^j(\lambda)=\begin{cases}
0 & \text{ if } \lambda \leq f^j(0) \\
(f^j)^{-1}(\lambda) & \text{ if } f^j(0) < \lambda \leq f^j(1) \\
1 & \text{ if } f^j(1)<1.
\end{cases}
\end{equation}
The next claim follows from the intermediate value theorem and the monotonicity of $g^j$ for each $j$.
\begin{claim}
\label{claim1}
There exists $\lambda \in \mathbb{R}$ such that $g(\lambda)=1$. Moreover, if there exist $\lambda_1, \lambda_2\in\mathbb{R}$ such that $g(\lambda_1)=g(\lambda_2)$, then $g^j(\lambda_1)=g^j(\lambda_2)$ for each $j$.
\end{claim}
\begin{claim}
\label{claim2}
If $Q^*\in\Delta^{J-1}$ and $\lambda\in\mathbb{R}$ satisfy \eqref{eqKKT}, then it is necessary that  $g(\lambda)=1$.
\end{claim}
\begin{proof}
Assume \eqref{eqKKT} and $g(\lambda)\neq 1$. Then
\begin{subequations}
\label{eqgchain}
\begin{align}
\sum\nolimits_j Q^{j*} &= \sum\nolimits_{j: Q^{j*}>0} Q^{j*} \label{eqgchain1}\\
&= \sum\nolimits_{j: Q^{j*}>0} g^j(\lambda) \label{eqgchain2}\\
&= \sum\nolimits_{j: Q^{j*}>0} g^j(\lambda) + \sum\nolimits_{j: Q^{j*}=0} g^j(\lambda) \label{eqgchain3} \\
&= \sum\nolimits_{j} g^j(\lambda) \\ 
&= g(\lambda)\neq 1.
\end{align}
\end{subequations}
Equality \eqref{eqgchain2} follows from \eqref{eqKKT1} and \eqref{eqgdef}. Equality \eqref{eqgchain3} holds since, for $j$ such that $Q^{j*}=0$, we have $f_N^j(0)\geq \lambda$ by \eqref{eqKKT2} and thus, by definition \eqref{eqgdef}, $g^j(\lambda)=0$. However, the chain of equalities \eqref{eqgchain} is a contradiction to $Q^*\in\Delta^{J-1}$.
\end{proof}

Now, pick any $\lambda\in\mathbb{R}$ such that $g(\lambda)=1$, and construct $Q^*$ by $Q^{j*}=g^j(\lambda)$. By Claim~\ref{claim1}, such a $Q^*$ is unique. It is easy to check that $Q^*\in\Delta^{J-1}$ and the condition \eqref{eqKKT} is satisfied. This construction, together with Claim~\ref{claim2}, shows that there exists a unique $Q^*\in\Delta^{J-1}$ satisfying the condition \eqref{eqKKT}.
\fi

\section*{Acknowledgment}
The authors would like to thank Mr. Matthew T. Morris and Mr. James S. Stanesic  at the University of Texas at Austin for their contributions to the numerical study in Section~\ref{secsimulation}.  The first author also acknowledges valuable discussions with Dr. Tamer Ba\c{s}ar at the University of Illinois at Urbana-Champaign.

\ifCLASSOPTIONcaptionsoff
  \newpage
\fi

\bibliographystyle{IEEEtran}
\bibliography{Refs}

%
%
%
%

\end{document}